\documentclass[11pt]{amsart}

\title{Counting Egyptian fractions}

\author{Sandro Bettin}
\address{Dipartimento di Matematica, Universit\`{a} di Genova, Via Dodecaneso 35, 16146 Genova, Italy}
\email{bettin@dima.unige.it}

\author{Lo\"{\i}c Greni\'{e}}
\address{Dipartimento di Ingegneria Gestionale, dell'Informazione e della Produzione, Universit\`{a} di Bergamo, viale Marconi 5, 24044 Dalmine, Italy}
\email{loic.grenie@gmail.com}

\author{Giuseppe Molteni}
\address{Dipartimento di Matematica, Universit\`{a} di Milano, Via Saldini 50, 20133 Milano, Italy}
\email{giuseppe.molteni1@unimi.it}

\author{Carlo Sanna}
\address{Dipartimento di Matematica, Universit\`{a} di Genova, Via Dodecaneso 35, 16146 Genova, Italy}
\email{carlo.sanna.dev@gmail.com}

\usepackage{amsmath}
\usepackage{amssymb}
\usepackage{amsthm}
\usepackage{geometry}
\usepackage{graphicx}
\usepackage{color}
\usepackage{array}
\usepackage{hyperref}
\usepackage{enumerate}
\hypersetup{colorlinks=true}
\usepackage{booktabs}
\usepackage{float}
\usepackage{soul}

\newtheorem{thm}{Theorem}[section]

\newtheorem{lem}[thm]{Lemma}
\theoremstyle{remark}
\newtheorem{rmk}{Remark}[section]

\newcommand{\N}{\mathbb{N}}
\newcommand{\Z}{\mathbb{Z}}
\newcommand{\R}{\mathbb{R}}
\newcommand{\m}{\mathfrak{m}}

\renewcommand{\SS}{\mathfrak{S}}
\newcommand{\EE}{\mathfrak{E}}
\newcommand{\af}{\mathfrak{a}}
\newcommand{\eps}{\varepsilon}
\newcommand{\lcm}{\operatorname{lcm}}
\renewcommand{\emptyset}{\varnothing}

\newcolumntype{L}{>{$}l<{$}}
\newcolumntype{R}{>{$}r<{$}}
\newcolumntype{C}{>{$}c<{$}}

\begin{document}

\begin{center}
\emph{Important:} It has been pointed out to the authors that the bounds on $\#\EE_N$ given in this paper were already proved in~\cite{MR0366795} (the upper bound in a stronger form).\\
We plan to update further this article at a later stage.
\end{center}
\vspace{2em}

\begin{abstract}
For any integer $N \geq 1$, let $\EE_N$ be the set of all Egyptian fractions employing denominators less
than or equal to $N$. We give upper and lower bounds for the cardinality of $\EE_N$, proving that
\begin{equation*}
\frac{N}{\log N} \prod_{j = 3}^{k} \log_j N<\log(\#\EE_N) < 0.421\, N,
\end{equation*}
for any fixed integer $k\geq3$ and every sufficiently large $N$, where $\log_j x$ denotes the $j$-th
iterated logarithm of $x$.
\end{abstract}

\maketitle

\section{Introduction}

Every positive rational number $a/b$ can be written in the form of an \emph{Egyptian fraction}, that is,
as a sum of distinct unit fractions: $a/b=1/n_1+\cdots+1/n_r$ with $n_1,\dots,n_r\in\N$ distinct. %
Several properties of these representations have been investigated. For example, it is known that all
rationals $a/b\in(0,1)$ are representable using only denominators which are
$O(b (\log b)^{1+\eps})$~\cite{MR852193,MR961916} (see also~\cite{MR1057319}) or that $O(\sqrt{\log b})$
different denominators are always sufficient~\cite{MR766441}. It is also well understood which integers
can be represented using denominators up to a bound $x$~\cite{MR1832627}, and
Martin~\cite{MR1608486,MR1793163} showed that any rational can be represented as a ``dense Egyptian
fraction''.
In this paper we take a different direction, and study the cardinality of the set of rational numbers
representable using denominators up to $N$,
\begin{equation*}
\EE_N := \bigg\{\sum_{n = 1}^N \frac{t_n}{n}\colon t_1, \dots, t_N \in \{0,1\} \bigg\},\qquad N\in\N,
\end{equation*}
as $N\to+\infty$.

Another motivation for studying the cardinality of $\EE_N$ comes from the recent work of three of the
authors~\cite{MR3907571} (see also~\cite{PreprintGreedy}), where the question of how well a real number
$\tau$ can be approximated by sums of the form $\sum_{n = 1}^N s_n / n$, where $s_1, \dots, s_N \in \{-1,
+1\}$, is studied. Precisely, let
\begin{equation*}
\SS_N := \left\{\sum_{n = 1}^N \frac{s_n}{n}\colon s_1, \dots, s_N \in \{-1,+1\} \right\}
\quad\text{and}\quad
\m_N(\tau) := \min\left\{|\tau - \sigma|\colon \sigma \in \SS_N \right\} ,
\end{equation*}
for every positive integer $N$. Note that $\SS_N$ and $\EE_N$ have the same cardinality, since $x \mapsto
2x - \sum_{n=1}^N 1/n$ is a bijection $\EE_N \to \SS_N$.

It has been proved that $\m_N(\tau) < \exp\big(-(\tfrac1{\log 4} - \eps)(\log N)^2\big)$
for every $\tau \in \R$, $\eps > 0$, and for all sufficiently large
positive integers $N$, depending on $\tau$ and $\eps$~\cite[Theorem~1.1]{MR3907571}; and that for any $f\colon \N \to \R^+$, there exists $\tau_f \in \R$ such that $\m_N(\tau_f) <
f(N)$ for infinitely many $N$~\cite[Proposition~5.9]{PreprintGreedy}. On the other hand, it is possible to obtain lower bounds for $\m_N(\tau)$
holding for almost all $\tau$ by giving upper bounds for the cardinality of $\EE_N$. Indeed, defining
\begin{equation*}
\alpha := \limsup_{N \to +\infty} \frac{\log(\#\EE_N)}{N} ,
\end{equation*}
by a Borel-Cantelli argument one obtains the following lower bound.

\begin{lem}\label{lem:almostevery}
Fix $\eps > 0$.
For almost all $\tau\in\R$ we have
\begin{equation*}
\m_N(\tau) > \exp\!\left(-(\alpha + \eps) N\right)
\end{equation*}
for all sufficiently large $N$, depending on $\tau$ and $\eps$.
\end{lem}

In~\cite[Proposition~2.7]{MR3907571} we proved that $\alpha < 0.6649$ with a simple argument, improving
upon the trivial $\alpha\leq \log 2=0.69314\dots$
In this paper, we improve this bound even further,
thus producing a better lower bound for $\m_N(\tau)$.

\begin{thm}\label{thm:upper}
We have $\alpha < 0.421$, that is, $\log(\#\EE_N) < 0.421 \, N$ for all $N$ large enough.
\end{thm}

In the opposite direction, it is not difficult to show that $\log(\#\EE_N) \gg N/\log N$. In the
following theorem we show that one can slightly improve over this lower bound.

\begin{thm}\label{thm:lower}
For every integer $k\geq 3$ we have
\begin{equation}\label{equ:logEElower}
\log(\#\EE_N) > \frac{N}{\log N} \prod_{j = 3}^{k} \log_j N
\end{equation}
for all sufficiently large $N$, depending on $k$.
\end{thm}
\noindent %
Theorem~\ref{thm:lower} is proved by showing that the set of integers $N$ for which the quotient
$\#\EE_N/\#\EE_{N-1}$ attains its maximum value $2$ is quite large. With some more effort it is possible
to give an explicit sequence of positive integers $(N_k)_{k \geq 1}$ such that \eqref{equ:logEElower}
holds for every $N \geq N_k$.

Upper and lower bounds in Theorems~\ref{thm:upper} and~\ref{thm:lower} have different orders and it is
not clear whether one of them seizes the true behavior of the sequence $\#\EE_N$. Only a few of these
numbers can be computed, since the algorithms devised for this purpose have an exponential behavior (in
time or in memory). The ones which are known are in Table~\ref{tab:1} and partially appear as
sequence A072207 of~\cite{OEIS}. A graph of $\log(\#\EE_N)/N$ and of $\log(\#\EE_N)/(N/\log N)$ is in
Figure~\ref{Fig:1}. %
As the range of $N$ for which $\#\EE_N$ is known is very small, the data do not make it clear whether
$\log(\#\EE_N)/N$ might converge to $0$ or to any other real number.
\begin{table}[H]
\begin{tabular}{RR|RR|RR|RR}
  \toprule
 N & \EE_N    &   N & \EE_N   &   N & \EE_N      &   N & \EE_N       \\
  \midrule
 1 &    2     &  12 &   1856  &  23 &    896512  &  34 &   224129024 \\
 2 &    4     &  13 &   3712  &  24 &    936832  &  35 &   231010304 \\
 3 &    8     &  14 &   7424  &  25 &   1873664  &  36 &   237031424 \\
 4 &   16     &  15 &   9664  &  26 &   3747328  &  37 &   474062848 \\
 5 &   32     &  16 &  19328  &  27 &   7494656  &  38 &   948125696 \\
 6 &   52     &  17 &  38656  &  28 &   7771136  &  39 &  1896251392 \\
 7 &  104     &  18 &  59264  &  29 &  15542272  &  40 &  1928593408 \\
 8 &  208     &  19 & 118528  &  30 &  15886336  &  41 &  3857186816 \\
 9 &  416     &  20 & 126976  &  31 &  31772672  &  42 &  3925999616 \\
10 &  832     &  21 & 224128  &  32 &  63545344  &  43 &  7851999232 \\
11 & 1664     &  22 & 448256  &  33 & 112064512  &     &             \\
  \bottomrule
\end{tabular}
\caption{The first values of $\#\EE_N$.}
\label{tab:1}
\end{table}

\begin{figure}[H]
\centering
\input{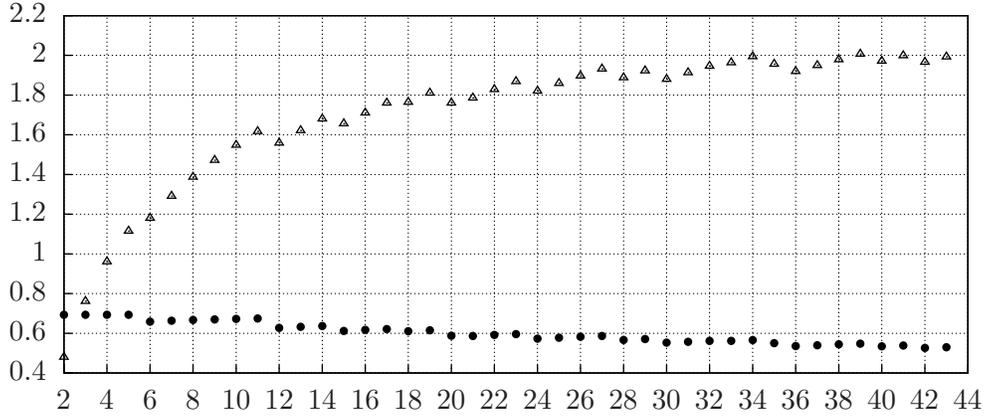}
\caption{Graph of $\log(\#\EE_N)/N$ (dots) and of $\log(\#\EE_N)/(N/\log N)$ (triangles) for $N \leq 43$.}
\label{Fig:1}
\end{figure}

\subsection*{Acknowledgements}
S.~Bettin is member of the INdAM group GNAMPA.
L.~Greni\'{e}, G.~Mol\-te\-ni and C.~Sanna are members of the INdAM group GNSAGA.
C.~Sanna is supported by a postdoctoral fellowship of INdAM.
The extensive computations needed for this paper have been performed on the UNITECH INDACO computing
platform of the Universit\`{a} di Milano and on the computing cluster of the Universit\'{e} de Bordeaux.
The authors warmly thank Alessio Alessi and Karim Belabas for their operative assistance for these
computations.

\subsection*{Notation}
We employ the Landau--Bachmann ``Big Oh'' notation $O$ as well as the associated Vinogradov symbols $\ll$
and $\gg$. We reserve the letter $p$ for prime numbers. We put $\log_1 x := \log x$ and $\log_{k+1} x :=
\log (\log_k x)$ for every integer $k \geq 1$ and every sufficiently large $x$.

\section{Proof of Lemma~\ref{lem:almostevery}}

The definition of $\alpha$ implies that we have an upper bound $\#\SS_N \leq e^{(\alpha + \eps/2) N}$,
for all large enough $N$. The claim follows by the Borel--Cantelli lemma. We have
\begin{align*}
\mathcal{E} &:= \{\tau \in \R\colon \m_N(\tau) \leq e^{-(\alpha+\eps) N} \text{ for infinitely many } N\} \\
&= \bigcap_{M = 1}^\infty \bigcup_{N \geq M} \{\tau \in \R\colon \m_N(\tau) \leq e^{-(\alpha+\eps) N}\} .
\end{align*}
Hence, the Lebesgue measure of $\mathcal{E}$ is estimated by
\begin{align*}
\operatorname{meas}(\mathcal{E})
\leq \inf_{M} \sum_{N \geq M} 2e^{-(\alpha+\eps) N} \#\SS_N
\leq \inf_{M} \sum_{N \geq M} 2e^{-\eps N/2}
= \inf_{M} \frac{2e^{-\eps M/2}}{1 - e^{-\eps/2}}
= 0.
\end{align*}
This implies that, for almost every $\tau$, the lower bound $\m_N(\tau) > e^{-(\alpha+\eps) N}$ holds for
all sufficiently large $N$.

\section{Proof of Theorem~\ref{thm:upper}}

For each prime number $p$ and for every positive integer $n$, let $\nu_p(n)$ denote the $p$-adic valuation of $n$.
We begin with the following easy lemma.

\begin{lem}\label{lem:density}
For each prime number $p$, let $\mu_p$ be a positive integer, and suppose that $\mu_p = 1$ for all but
finitely many primes $p$. Then, the natural density of the set
\begin{equation*}
D := \{n \in \mathbb{N}\colon \mu_p \mid \nu_p(n) \text{ for every prime $p$}\}
\end{equation*}
is equal to
\begin{equation*}
\delta := \prod_p \frac{1 - p^{-1}}{1 - p^{-\mu_p}} ,
\end{equation*}
where we note that only finitely many prime numbers contribute to the product.
The formula can also be written as
\begin{equation*}
\delta
= \Big(\sum_{d|M}\frac{1}{d}\Big)^{-1}
= M\Big(\sum_{d|M} d \Big)^{-1},
\end{equation*}
where $M:= \prod_{p} p^{\mu_p-1}$.
\end{lem}
\begin{proof}
This is a standard argument, we reproduce it here for completeness. Let $\delta_D(n):=1$ when $n\in
D$ and $\delta_D(n):=0$ otherwise. Note that $\delta_D$ is a multiplicative function, so that
\begin{align*}
F(s) &:= \sum_{n=1}^{\infty} \frac{\delta_D(n)}{n^s}
      = \prod_p\left(1 + \sum_{b=1}^{\infty} \frac{\delta_D(p^{b})}{p^{bs}}\right)
      = \prod_p\left(1 + \sum_{b=1}^{\infty} \frac{1}{p^{\mu_p bs}}\right)\\
     &= \prod_p\left(1 - \frac{1}{p^{\mu_p s}}\right)^{-1}
      = \zeta(s)\prod_p\frac{1 - p^{-s}}{1 - p^{-\mu_p s}}
      =: \zeta(s)H(s).
\end{align*}
Note that $H(s)$ is a finite Euler product (because $\mu_p=1$ for a.e.~prime), which is analytic for
$\Re(s)>0$. Set $h\colon \N\to \R$, with $H(s)=:\sum_{n=1}^{\infty} h(n)n^{-s}$. Then
\begin{align*}
\sum_{n\leq x} \delta_D(n)
&= \sum_{n\leq x} \sum_{d \mid n}h(d)
= \sum_{d\leq x} h(d) \sum_{n\leq x/d} 1
= \sum_{d\leq x} h(d) \left\lfloor{\frac{x}{d}}\right\rfloor
= x\sum_{d\leq x} \frac{h(d)}{d}
+ O\!\left(\sum_{d\leq x} |h(d)|\right)\\
&= xH(1) + O\!\left(x\sum_{d>x} \frac{|h(d)|}{d}\right) + O\!\left(\sum_{d\leq x} |h(d)|\right)
 = xH(1) + O_\eps(x^\eps)
\end{align*}
and $\delta=H(1)$.
(A more careful analysis shows that the error term has size $\ll (\frac{1}{\ell}\log x)^\ell$, uniformly
in $\ell\geq 1$ and $x\geq 2$.)
The alternative representation of $\delta$ as a sum of divisors of $M$
follows immediately by the unique factorization of integers as a product of prime powers.
\end{proof}

The next lemma is our key tool to provide numerical upper bounds for $\alpha$.

\begin{lem}\label{lem:chains}
Let $\af_1 \subset \cdots \subset \af_\ell =:\af$ be finite nonempty sets of natural numbers. We have
\begin{equation*}
\alpha
\leq \log 2 - \delta \sum_{i = 1}^\ell \left(\frac{1}{\max(\af_{i})} - \frac{1}{\max(\af_{i + 1})}\right) \log\!\left(\frac{2^{\#\af_i}}{r_i}\right) ,
\end{equation*}
where $1/\max(\af_{\ell + 1}) := 0$,
\begin{equation*}
\delta := \Big(\sum_{d|M} \frac{1}{d} \Big)^{-1}
\quad\text{with}\quad
M:= \lcm\{a\colon a\in \af\},
\end{equation*}
and, for $i = 1,\dots,\ell$,
\begin{equation*}
r_i := \# \! R_i
\quad\text{with}\quad
R_i := \left\{\sum_{a \in \af_i} \frac{t_a}{a}\colon t_a \in \{0, 1\}\right\}.
\end{equation*}
\end{lem}
\begin{proof}
For any any prime $p$ let $\mu_p := 1 + \max\{\nu_p(a)\colon a \in \af\}$ and
\begin{equation*}
D := \{n \in \mathbb{N}\colon \mu_p \mid \nu_p(n) \text{ for every prime $p$}\} .
\end{equation*}
Let $k \af := \{ka\colon a \in \af\}$ for every integer $k$.
Note that $k_1 \af \cap k_2 \af = \emptyset$ for all $k_1, k_2 \in D$ with $k_1 \neq k_2$.
Indeed, suppose that $k_1 a_1 = k_2 a_2$ for some $k_1, k_2 \in D$ and $a_1, a_2 \in \af$.
Then, for every prime number $p$, we have
\begin{equation*}
\nu_p(a_1) - \nu_p(a_2) \equiv \nu_p(k_2) - \nu_p(k_1) \equiv 0 \pmod{\mu_p} ,
\end{equation*}
which in turn implies that $\nu_p(a_1) = \nu_p(a_2)$. Hence, $a_1 = a_2$ and $k_1 = k_2$.

Clearly $\mu_p$ is $1$ for all but finitely many prime numbers $p$. Hence, Lemma~\ref{lem:density}
applies, and the natural density of the set $D$ is
\begin{equation*}
\delta = \prod_p \frac{1 - p^{-1}}{1 - p^{-\mu_p}}
= \Big(\sum_{d|M} \frac{1}{d} \Big)^{-1},
\end{equation*}
with $M= \prod_{p} p^{\mu_p-1} =\lcm\{a\colon a\in \af\}$.
Let $N$ be a positive integer. Since $\af_1,\dots,\af_\ell$ are contained in $\af$, we have that the
sets $k \af_i$, with $i \in \{1,\ldots,\ell\}$ and $k \in D \cap \left(N / \max(\af_{i + 1}), N /
\max(\af_{i})\right]$ are pairwise disjoint. Let $F_N$ be their union. Clearly, $F_N \subseteq
\{1,\dots,N\}$. Also, we have
\begin{align}\label{equ:Fnsize}
\#F_N
&= \sum_{i = 1}^\ell \#\!\left(D \cap \left(\frac{N}{\max(\af_{i + 1})}, \frac{N}{\max(\af_{i})}\right]\right) \#\af_i \\
&= (\delta + o(1)) N \sum_{i = 1}^\ell \left(\frac{1}{\max(\af_{i})} - \frac{1}{\max(\af_{i + 1})}\right) \#\af_i , \nonumber
\end{align}
as $N \to +\infty$.

We are finally ready to give an upper bound for $\#\EE_N$. Every element of $\EE_N$ is of the form
\begin{equation*}
\sum_{i = 1}^\ell \sum_{k \in D \cap \left(\frac{N}{k \in \max(\af_{i + 1})}, \frac{N}{\max(\af_{i})}\right]} \sum_{a \in k\af_i} \frac{t_a}{a}
+ \sum_{b \in F_N^\prime} \frac{t_b}{b} ,
\end{equation*}
where $t_n \in \{0, 1\}$ for all $n \in \{1, \dots, N\}$, and $F_N^\prime := \{1, \dots, N\} \setminus
F_N$. Therefore, we have
\begin{align*}
\log \EE_N &\leq \sum_{i = 1}^\ell \#\!\left(D \cap \left(\frac{N}{\max(\af_{i + 1})}, \frac{N}{\max(\af_i)}\right]\right) \log r_i + \log (2^{N - \#F_N}) \\
&= N \left(\log 2 + (\delta + o(1)) \sum_{i = 1}^\ell \left(\frac{1}{\max(\af_{i})} - \frac{1}{\max(\af_{i + 1})}\right) \log r_i\right) - \#F_N \log 2 \\
&= N \left(\log 2 - (\delta + o(1)) \sum_{i = 1}^\ell \left(\frac{1}{\max(\af_{i})} - \frac{1}{\max(\af_{i + 1})}\right) \log\!\left(\frac{2^{\#\af_i}}{r_i}\right) \right) ,
\end{align*}
as $N \to +\infty$, where we used \eqref{equ:Fnsize}.
Consequently,
\begin{equation*}
\alpha \leq \log 2 - \delta \sum_{i = 1}^\ell \left(\frac{1}{\max(\af_{i})} - \frac{1}{\max(\af_{i + 1})}\right) \log\!\left(\frac{2^{\#\af_i}}{r_i}\right) ,
\end{equation*}
as claimed.
\end{proof}

Lemma~\ref{lem:chains} gives non-trivial bounds for $\alpha$ already when applied in the simplest case
$\ell=1$.
Table~\ref{tab:2} displays, for some choices of $\af=\af_1$, the value of the parameter $r=r_1$ (obtained
numerically) and the corresponding bound $\alpha\leq \log 2 - \frac{\delta}{\max(\af)}
\log({2^{\#\af}}/{r})$.

\begin{table}[H]
\begin{tabular}{LLLL}
  \toprule
   \af                                                 & r         &\delta/\max(\af) & \alpha \leq \\
  \midrule
\{1,2,3,6\}                                            & 13        & 1/12            & 0.67584390\\
\{1,2,3,4,6,12\}                                       & 29        & 1/28            & 0.66487620\\      
\{1,2,3,4,5,6,8,10,12,15\}                             & 302       & 1/45            & 0.66601285\\
\{1,2,3,4,5,6,10,12,15,20\}                            & 162       & 1/56            & 0.66022083\\
\{1,2,3,4,5,6,8,10,12,15,16,20,24,30\}                 & 694       & 1/93            & 0.65915160\\
\{1,2,3,4,5,6,8,9,10,12,15,18,20,24,30\}               & 1061      & 2/195           & 0.65796522\\
\{1,2,3,4,5,6,7,8,9,10,12,14,15,18,20,21,24,28,30\}    & 7757      & 7/780           & 0.65533420\\
  \bottomrule
\end{tabular}
\caption{Upper bounds for $\alpha$ in the case $\ell=1$.}
\label{tab:2}
\end{table}

Notice that it is always convenient to include $1$ in $\af$ since its presence does not affect the fraction
$\delta/\max(\af)$ whereas it typically increases the value of $\log({2^{\#\af}}/{r})$. The result
in~\cite[Proposition~2.7]{MR3907571} corresponds to the above construction with $\af = \{1,2,3,4,6,12\}$.
A judicious choice of a larger $\af$ improves the result, but the gain in $\log({2^{\#\af}}/{r})$ is
considerably tempered by the size of $\delta/\max(\af)$, which becomes smaller and smaller. The last
entry in Table~\ref{tab:2} represents the best bound we were able to obtain with this approach.

Better results can be attained by taking $\ell>1$. Given $1\leq a_1< a_2< \cdots <a_\ell$, one can
take the collection $\af_1\subset \af_2\subset \ldots\subset \af_\ell=:\af$
with $\af_i=\{a_1,\ldots, a_i\}$ for $i=1,\dots,\ell$. In this case the bound in Lemma~\ref{lem:chains} simplifies to
\begin{equation*}
\alpha
\leq \alpha_1 := \log 2 - \delta \sum_{i = 1}^\ell \left(\frac{1}{a_{i}} - \frac{1}{a_{i + 1}}\right) \log\!\left(\frac{2^i}{r_i}\right) ,
\end{equation*}
with $1 / a_{\ell + 1} := 0$.
A first naive choice is to take $\af_i:=\{1,\ldots, i\}$ for $1\leq i\leq \ell$. %
Notice that with this choice $r_i=\#\EE_i$. Now, for some $I\leq \ell$ let $\beta>0$ be such that
$r_i\leq e^{\beta i}$ for $I\leq i\leq \ell$. Then
\begin{align*}
\alpha_1 &= \log 2 - \delta \sum_{i = I}^{\ell} \Big(\frac{1}{i} - \frac{1}{i+1}\Big) i(\log 2-\beta) +O(\delta\log(1+|I|))\\[-0.5em]
&\leq \log 2 - \delta (\log 2-\beta) \sum_{i = 1}^{\ell} \frac{1}{i}+O(\delta\log(1+|I|)).
\end{align*}
Now
\begin{equation*}
\delta=\prod_{p\leq \ell} \frac{1 - p^{-1}}{1 - p^{-\mu_{p,\ell}}}
\end{equation*}
with $\mu_{p,\ell}=1+\max\{t \colon p^t\leq\ell\}=1+[\frac{\log \ell}{\log p}]\geq \frac{\log \ell}{\log p}.$
Thus, by Merten's theorem and since
\begin{align*}
\sum_{p\leq \ell}\log(1 - p^{-\mu_{p,\ell}})=\sum_{p\leq \ell}\log(1 - \ell^{-1})\ll 1/\log( \ell+1),
\end{align*}
one obtains $\delta\sim e^{-\gamma}/\log \ell$ as $\ell\to\infty$.
In particular, taking for example
$I=[e^{\sqrt{\log \ell}}]$, we get
\begin{equation*}
\alpha_1 \leq
\log 2 - e^{-\gamma} (\log 2-\beta)+o(1)
\end{equation*}
as $\ell\to\infty$ and so for any $\eps>0$ we obtain $\alpha_1<\log 2 - e^{-\gamma} (\log 2-\beta)+\eps$
if $\ell$ is large enough. By the definition of $\alpha$ we have that for all $\eps>0$ one has $\beta\leq
\alpha+\eps$ for large enough $\ell$, thus giving $\alpha_1 < \log 2 - e^{-\gamma} (\log 2-\alpha-\eps)$.
Notice that due to the loss of the factor $e^{-\gamma}$ this argument fails to recover an upper bound of
the type $\alpha+\eps$ for $\alpha_1$, and for example it would only gives $\alpha_1 \leq
(1-e^{-\gamma})\log 2 + \eps = 0.3039\dots+\eps$ if $\alpha=0$. However, it shows that by taking $\ell$
large enough one can surely get arbitrarily close to the upper bound $\log 2 - e^{-\gamma} (\log
2-\alpha)$ just by performing a finite numerical computation.

An alternative and more effective approach arises from the observation that the density $\delta$ depends only
on the valuations $\mu_p$, and so it is the same for different $\af$ sharing the same least common multiple.
In particular, it is convenient to select the numbers in $\af$ as the full collection of divisors of a
given integer $M$.
In this case the bound further simplifies to
\begin{equation}\label{equ:alpha1fulldivisors}
\alpha
\leq \alpha_1
\qquad\text{with}\qquad
\alpha_1 = \delta \sum_{i = 1}^\ell \left(\frac{1}{a_{i}} - \frac{1}{a_{i + 1}}\right) \log r_i .
\end{equation}
Indeed, if $a_1 < \cdots < a_\ell$ are all the divisors of $M$, then we get
\begin{align*}
\log 2 - \delta \sum_{i = 1}^\ell \left(\frac{1}{a_{i}} - \frac{1}{a_{i + 1}}\right) \log(2^i)
&= \left(1 - \delta \sum_{i = 1}^\ell \left(\frac{1}{a_{i}} - \frac{1}{a_{i + 1}}\right)i\right) \log 2\\
&= \left(1 - \delta \sum_{i = 1}^\ell \frac{1}{a_{i}}\right) \log 2
 = \left(1 - \delta \sum_{d|M} \frac{1}{d}\right) \log 2
 = 0.
\end{align*}
Numerically it seems that best results come from taking $M$ with only ``small'' prime
divisors (say, for example, $M=\lcm\{1,\dots,m\}$, for some $m$, or some small variations of it) and this
could be explained by the observation that with numbers of this shape an argument analogous to the above
does not have a loss of $e^{-\gamma}$ (notice however that in this case $r_i$ is not equal to $\#\EE_i$,
so the argument is not fully recursive as in the previous case).
Table~\ref{tab:3} collects some bounds
one gets in this way.
Most likely further improvements could be obtained by taking larger numbers.
However, the quantity of time
and memory one needs to compute $r_i$ for large $i$ prevented us to any significant improvement on the
value appearing at the bottom of Table~\ref{tab:3}.
\begin{table}[H]
\begin{tabular}{LL}
  \toprule
  M                                                                          & \alpha \leq \\
  \midrule
5040= 2^4\cdot 3^2\cdot 5\cdot 7                                             & 0.56731289\\
%
%
%
529200= 2^4\cdot 3^3\cdot 5^2\cdot 7^2                                       & 0.55084208\\
55440= 2^4\cdot 3^2\cdot 5\cdot 7\cdot 11                                    & 0.54939823\\
%
%
2116800= 2^6\cdot 3^3\cdot 5^2\cdot 7^2                                      & 0.54552567\\
4233600= 2^7\cdot 3^3\cdot 5^2\cdot 7^2                                      & 0.54465996\\
%
%
1441440= 2^5\cdot 3^2\cdot 5\cdot 7\cdot 11\cdot 13                          &  0.53020542\\
2162160= 2^4\cdot 3^3\cdot 5\cdot 7\cdot 11\cdot 13                          &  0.52779949\\
%
%
4324320= 2^5\cdot 3^3\cdot 5\cdot 7\cdot 11\cdot 13                          &  0.52405384\\
%
%
%
%
%
%
43243200= 2^6\cdot 3^3\cdot 5^2\cdot 7\cdot 11\cdot 13                       & 0.51452256\\ 
%
%
147026880= 2^6\cdot 3^3\cdot 5\cdot 7\cdot 11\cdot 13\cdot 17                & 0.51032288\\ 
%
%
%
2793510720= 2^5\cdot 3^3\cdot 5\cdot 7\cdot 11\cdot 13\cdot 17\cdot 19       & 0.49944226\\ 
%
%
13967553600= 2^6\cdot 3^3\cdot 5^2\cdot 7\cdot 11\cdot 13\cdot 17\cdot 19    & 0.49153796\\ 
41902660800= 2^6\cdot 3^4\cdot 5^2\cdot 7\cdot 11\cdot 13\cdot 17\cdot 19    & 0.48948987\\ 
  \bottomrule
\end{tabular}
\caption{Upper bounds for $\alpha$ using increasing collections $\af_j:=\{a_1,\ldots, a_j\}$ for
$j=1,\ldots,\ell$ and $\af=\{d\colon d|M\}$ for different $M$.}
\label{tab:3}
\end{table}

The search of a good way to store the huge vectors containing the numbers in $R_i$ led us to consider
the following upper bound for $r_i$.
\begin{lem}\label{lem:A}
Let $\af:=\{a_1<a_2<\cdots<a_\ell\}$ be the full set of divisors of an integer $M$, let
$\af_i:=\{a_1,\ldots,a_i\}$, and let $R_i$ and $r_i$ as in Lemma~\ref{lem:chains}. Then
\begin{equation*}
r_i \leq 1 + \lcm\{a_1,\ldots,a_i\}\sum_{k=1}^i \frac{1}{a_k}
\end{equation*}
\end{lem}
\begin{proof}
Let $L_i := \lcm\{a_1, \dots, a_i\}$.
Hence, $L_i/a_k\in \Z$ for $k=1,\dots,i$.
Consequently, $L_iR_i = \{\sum_{k=1}^i \pm L_i/a_k \}\subseteq \Z$. Moreover, $\max\{L_iR_i\} =
-\min\{L_iR_i\} = M_i := \sum_{k=1}^i L_i/a_k$, so that $\#(L_iR_i)\leq 1+2M_i$. We can improve this
bound by a factor $2$ since $\pm 1\equiv 1\pmod 2$, so that all numbers in $L_iR_i$ have the same parity of
$M_i$. As a consequence, $r_i = \#R_i = \#(L_iR_i)\leq 1+M_i$, as desired.
\end{proof}

At this point, we are ready to explain the method that we used to prove the bound of
Theorem~\ref{thm:upper}. The main idea is to use the upper bound given by~\eqref{equ:alpha1fulldivisors},
but computing the exact value of $r_i$ only for small $i$. Precisely, let $M$ be a (very large) positive
integer and let $a_1 < \cdots < a_\ell$ be all its divisors. Also, let $M^\prime$ be a (small) divisor of
$M$ and let $a_1^\prime < \cdots < a_m^\prime$ be all its divisors. %
For each divisor $a_j^\prime$ of $M^\prime$ we pre-compute the value $r_j^\prime := \left\{\sum_{a \in
\af_j^\prime} \frac{t_a}{a} : t_a \in \{0,1\}\right\}$, where $\af_j^\prime := \{a_1^\prime, \dots,
a_j^\prime\}$. Then, for each divisor $a_i$ of $M$ we look for the largest $a_j^\prime$ dividing $a_i$,
and we estimate $r_i$ with the minimum between what we get from Lemma~\ref{lem:A} and the number
$r_j^\prime\cdot 2^{\#\af_i-\#\af_j^\prime}$. %
This is a correct bound since, by the assumption on $j$, we have $\af_j^\prime \subseteq \af_i$ and
consequently $r_{i}$ is at most $r_j^\prime$ multiplied by the power of two elevated to the difference of
the cardinalities of $\af_i$ and $\af_j'$.

Table~\ref{tab:5} collects some results we get in this way for suitable choices of $M^\prime$ and $M$.
The last entry gives our best result, which proves Theorem~\ref{thm:upper}.
Its computation needed approximatively 40 hours and 200GB.

\begin{table}[H]
\begin{tabular}{LLL}
  \toprule
  M                         & M^\prime                                  & \alpha \leq \\
  \midrule
\lcm\{1,\dots, 17\}         & \lcm\{1,\dots, 10\}                       & 0.52408774\\
\lcm\{1,\dots, 31\}         & \lcm\{1,\dots, 17\}                       & 0.47705355\\
\lcm\{1,\dots, 41\}         & \lcm\{1,\dots, 17\}                       & 0.46441548\\
\lcm\{1,\dots, 53\}         & \lcm\{1,\dots, 17\}                       & 0.44930098\\
\lcm\{1,\dots, 79\}         & \lcm\{1,\dots, 19\}                       & 0.43238142\\
\lcm\{1,\dots, 89\}         & \lcm\{1,\dots, 19\}                       & 0.42691423\\
\lcm\{1,\dots, 97\}         & \lcm\{1,\dots, 23\}                       & 0.42447521\\
\lcm\{1,\dots, 97\}         & 2\cdot3\cdot\lcm\{1,\dots, 19\}           & 0.42310594\\
\lcm\{1,\dots, 97\}         & 2^2\cdot3^2\cdot\lcm\{1,\dots, 19\}       & 0.42286665\\
\lcm\{1,\dots, 97\}         & 2^2\cdot3^2\cdot5\cdot\lcm\{1,\dots, 19\} & 0.42099405\\
  \bottomrule
\end{tabular}
\caption{Upper bounds for $\alpha$ mixing the computation of the true value of $r_i$ for small divisors, and
estimating its value with Lemma~\ref{lem:A} for large divisors. The relevant parameters are $M$ and $M^\prime$.}
\label{tab:5}
\end{table}

\section{Proof of Theorem~\ref{thm:lower}}\label{sec:lower}

Define the set
\begin{equation*}
\mathcal{U} := \left\{N \geq 1\colon \sum_{n=1}^{N-1} \frac{w_n}{n}
\neq \frac{1}{N}\colon \forall w_1, \dots, w_{N-1} \in \{-1,0,+1\}\right\} ,
\end{equation*}
and let $\mathcal{U}(x) := \mathcal{U} \cap [1, x]$ for every $x \geq 1$.

\begin{lem}\label{lem:doubling}
If $N \in \mathcal{U}$ then $\#\EE_N = 2\#\EE_{N - 1}$. Consequently, $\#\EE_N \geq
2^{\#\mathcal{U}(N)}$.
\end{lem}
\begin{proof}
On the one hand, $N \in \mathcal{U}$ implies that $\EE_{N - 1} \cap (\EE_{N - 1} + 1 / N) = \emptyset$.
On the other hand, $\EE_N = \EE_{N - 1} \cup (\EE_{N - 1} + 1/N)$.
Hence, we have $\#\EE_N = \#\EE_{N - 1} + \#(\EE_{N - 1} + 1/N) = 2\#\EE_{N - 1}$, as claimed.
\end{proof}

In light of Lemma~\ref{lem:doubling}, to produce a lower bound for $\#\EE_N$ it is sufficient
to give a lower bound for $\#\mathcal{U}(N)$.

\begin{lem}\label{lem:primes}
$\mathcal{U}$ contains $1$ and all prime numbers.
\end{lem}
\begin{proof}
The fact that $1\in\mathcal{U}$ follows immediately from the definition of $\mathcal{U}$. Furthermore,
for every prime number $p$ and for every $w_1, \dots, w_{p-1} \in \{-1,0,+1\}$, the identity
\begin{equation*}
\sum_{n=1}^{p-1} \frac{w_n}{n} = \frac{1}{p}
\end{equation*}
is impossible, since the left-hand side has nonnegative $p$-adic valuation (all denominators are not
divisible by $p$) while the right-hand side has negative $p$-adic valuation. Hence, $p \in \mathcal{U}$.
\end{proof}

For each positive integer $m$, let $d_m := \lcm\{1, \dots, m\}$ and $g_m := d_m \sum_{j = 1}^m 1 / j$.

\begin{lem}\label{lem:mpinUgm}
If $m \in \mathcal{U}$ and $p > g_m$ is a prime number, then $mp \in \mathcal{U}$.
\end{lem}
\begin{proof}
Suppose by contradiction that $N := mp \notin \mathcal{U}$. %
Hence, there exist $w_1, \dots, w_{N - 1} \in \{-1,0,+1\}$ such that
\begin{equation*}
\sum_{n=1}^{N-1} \frac{w_n}{n} = \frac{1}{N} .
\end{equation*}
Consequently, splitting the sum according to whether $p|n$ and multiplying by $p$ we obtain
\begin{equation*}
p\sum_{\substack{n=1 \\ p \nmid n}}^{N-1} \frac{w_n}{n} + \sum_{j = 1}^{m - 1} \frac{w_{pj}}{j} = \frac{1}{m} ,
\end{equation*}
which in turn implies that
\begin{equation}\label{equ:wpj}
\sum_{j = 1}^{m - 1} \frac{w_{pj}}{j} - \frac{1}{m} \equiv 0 \pmod p .
\end{equation}
Since $m \in \mathcal{U}$, the number on the left hand side is non-zero. Moreover, the absolute value of
its numerator is at most $g_m$, which by hypothesis is strictly smaller than $p$. But then
\eqref{equ:wpj} is impossible.
\end{proof}

\begin{rmk}
The proofs of Lemma~\ref{lem:primes} and Lemma~\ref{lem:mpinU} can be easily adapted to show that if $m
\in \mathcal{U}$ and $p > g_m$ is a prime number then $mp^k \in \mathcal{U}$ for every positive integer
$k$. However, this generalization does not lead to an improvement of our final result.
\end{rmk}

\begin{lem}\label{lem:mpinU}
If $m \in \mathcal{U}$ and $p > 3^m$ is a prime number, then $mp \in \mathcal{U}$.
\end{lem}
\begin{proof}
In light of Lemma~\ref{lem:mpinUgm}, it is enough to prove that $g_m < 3^m$ for every positive integer $m$.
On the one hand, from~\cite[p.~228]{MR0003018} we know that $d_m = \exp(\psi(m)) \leq \exp(1.04 m)$.
On the other hand,
\begin{equation*}
\sum_{j=1}^m \frac{1}{j} \leq 1 + \int_1^m \frac{\mathrm{d}t}{t} = \log(em) .
\end{equation*}
Hence, $g_m \leq \exp(1.04 m) \log(em) < 3^m$ for every integer $m \geq 25$.
A direct computation shows that the $g_m < 3^m$ holds also for $m=1,\dots,24$.
\end{proof}

As usual, let $\pi(x)$ denote the number of prime numbers not exceeding $x$.
The next lemma gives a recursive lower bound for $\#\mathcal{U}(x)$.

\begin{lem}\label{lem:recbound}
Let $y\geq 1$ and $x \geq 3^y$. Then
\begin{equation*}
\#\mathcal{U}(x) \geq \sum_{m \in \mathcal{U}(y)} \pi\!\left(\frac{x}{m}\right) - 2\cdot 3^y .
\end{equation*}
\end{lem}
\begin{proof}
Let us consider the natural numbers of the form $mp$, where $p$ is a prime number satisfying $3^y < p
\leq x/m$ and $m \in \mathcal{U}(y)$. Thanks to Lemma~\ref{lem:mpinU}, we have that $mp \in
\mathcal{U}(x)$. Moreover, these numbers can be written in the form $mp$ in a unique way. Indeed, for the
sake of contradiction, suppose that $mp=m^\prime p^\prime$ for some $m^\prime, p^\prime$ satisfying the
same conditions as $m, p$, with $p \neq p^\prime$. Then, $p^\prime \mid m$, so that $3^y < p^\prime \leq
m \leq y$, which is impossible. At this point, counting the choices for $m$ and $p$, we get
\begin{equation*}
\#\mathcal{U}(x)
\geq \sum_{m \in \mathcal{U}(y)} \left(\pi\!\left(\frac{x}{m}\right) - \pi(3^y)\right)
\geq \sum_{m \in \mathcal{U}(y)} \pi\!\left(\frac{x}{m}\right) - \pi(3^y) y .
\end{equation*}
The desired claim follows by applying the inequality $\pi(3^y) y \leq 2\cdot 3^y$, which in turn follows
from the estimate $\pi(x)\leq 2x/\log x$ valid for all $x\geq 2$.
\end{proof}

We need the following technical lemma. %
Let $e_1:= 1$ and $e_{k+1} := e^{e_k}$ for all integers $k \geq 2$.

\begin{lem}\label{lem:logintegral}
For any fixed integer $k \geq 2$, we have
\begin{equation*}
\int_{e_{k}}^x \frac{1}{t \log t} \prod_{j=3}^{k} \log_j t \,\mathrm{d}t \sim \prod_{j=2}^{k} \log_j x ,
\end{equation*}
as $x \to +\infty$.
\end{lem}
\begin{proof}
For $k=2$ the empty product appearing on the left hand side is set to $1$, by definition, and the
claim is clear in this case. Assume $k\geq 3$. We have
\begin{align*}
\left(\prod_{j=2}^{k} \log_j x\right)^\prime
 &= \left(\prod_{j=2}^{k} \log_j x\right) \sum_{j=2}^{k} \prod_{i = 0}^{j} \frac{1}{\log_i x} \\
 &= \frac{1}{x \log x} \left(\prod_{j=3}^{k} \log_j x\right) \left(1 + \sum_{j=3}^{k} \prod_{i = 3}^{j} \frac{1}{\log_j x}\right)
 \sim \frac{1}{x \log x} \left(\prod_{j=3}^{k} \log_j x\right) , \nonumber
\end{align*}
as $x \to +\infty$. The claim then follows from de~l'H\^{o}pital's rule.
\end{proof}

\begin{lem}\label{lem:Uxlower}
For every integer $k \geq 2$, we have
\begin{equation*}
\#\mathcal{U}(x) \gg_k \frac{x}{\log x} \prod_{j = 3}^{k} \log_j x ,
\end{equation*}
for all sufficiently large $x$, depending on $k$.
\end{lem}
\begin{proof}
We proceed by induction on $k$. By Lemma~\ref{lem:primes} and by Chebyshev's estimate, we have
\begin{equation*}
\#\mathcal{U}(x) > \pi(x) \gg \frac{x}{\log x} ,
\end{equation*}
for all sufficiently large $x$. This proves the claim for $k = 2$. Suppose $k \geq 3$ and that we have
already proved the claim for $k - 1$. Put $y := \tfrac1{2}\log x$ and assume that $x$ is sufficiently large. By
Lemma~\ref{lem:recbound} and by Chebyshev's estimate, we have
\begin{equation}\label{equ:Uthebound1}
\#\mathcal{U}(x)
\geq \sum_{m \in \mathcal{U}(y)} \pi\!\left(\frac{x}{m}\right) - 2\cdot 3^y
\gg \sum_{\substack{m \in \mathcal{U}(y)\\ m\geq 2}} \frac{x}{m \log (x / m)}
> \frac{x}{\log x} \sum_{\substack{m \in \mathcal{U}(y)\\ m\geq 2}} \frac{1}{m} .
\end{equation}
Furthermore, by partial summation and by the induction hypothesis, we get
\begin{align}\label{equ:Uthebound2}
\sum_{\substack{m \in \mathcal{U}(y)\\ m\geq 2}} \frac{1}{m}
&\gg \int_{e_{k-1}}^y \frac{\#\mathcal{U}(t)}{t^2} \,\mathrm{d}t
\gg_k \int_{e_{k-1}}^y \frac{1}{t \log t} \prod_{j=3}^{k-1}\log_j t \, \mathrm{d} t\gg \prod_{j=2}^{k-1} \log_j y \gg \prod_{j=3}^{k} \log_j x ,
\end{align}
for all sufficiently large $x$, depending on $k$, where we employed Lemma~\ref{lem:logintegral}.
Putting together \eqref{equ:Uthebound1} and \eqref{equ:Uthebound2} we obtain the desired claim.
\end{proof}

Theorem~\ref{thm:lower} follows easily by Lemma~\ref{lem:doubling} and~\ref{lem:Uxlower}.

\bibliographystyle{amsplain}

\end{document}